\documentclass[11pt,reqno]{amsart}
\usepackage{amscd,amsmath,amssymb,amsthm,amsfonts,epsfig,graphics}
\usepackage{geometry}
\usepackage{graphicx}
\usepackage{mathrsfs,amssymb}
\usepackage{bm}
\usepackage{url}
\theoremstyle{plain}

\newtheorem{corollary}{Corollary}

\newtheorem{lemma}{Lemma}

\newtheorem{proposition}{Proposition}
\newtheorem{remark}{Remark}

\newtheorem{theorem}{Theorem}
\numberwithin{equation}{section}

\allowdisplaybreaks[4]
\begin{document}
\title[Zeros of the deformed exponential function]
 {Zeros of the deformed exponential function}

\author{Liuquan Wang and Cheng Zhang}

\address{Division of Mathematics, Nanyang Technological University, 21 Nanyang Link, Singapore 637371, Singapore}
\email{mathlqwang@163.com;wanglq@ntu.edu.sg}

\address{Department of Mathematics, Johns Hopkins University, Baltimore, MD 21218, United States}
\email{czhang67@jhu.edu}

\subjclass[2010]{Primary 30C15, 11B83, 41A60; Secondary 11M36, 34K06}

\keywords{Deformed exponential function; asymptotic expansion; Eisenstein series; Bernoulli numbers}

\dedicatory{}

\maketitle

\begin{abstract}
Let $f(x)=\sum_{n=0}^{\infty}\frac{1}{n!}q^{n(n-1)/2}x^n$  ($0<q<1$) be the deformed exponential function. It is known that the zeros of $f(x)$ are real and form a negative decreasing sequence $(x_k)$ ($k\ge 1$). We investigate the complete asymptotic expansion for $x_{k}$ and prove that for any $n\ge1$, as $k\to \infty$,
\begin{align*}
x_k=-kq^{1-k}\Big(1+\sum_{i=1}^{n}C_i(q)k^{-1-i}+o(k^{-1-n})\Big),
\end{align*}
where $C_i(q)$ are some $q$ series which can be determined recursively. We show that each $C_{i}(q)\in \mathbb{Q}[A_0,A_1,A_2]$, where $A_{i}=\sum_{m=1}^{\infty}m^i\sigma(m)q^m$ and $\sigma(m)$ denotes the sum of positive divisors of $m$. When writing $C_{i}$ as a polynomial in $A_0, A_1$ and $A_2$, we find explicit formulas for the coefficients of the linear terms by using Bernoulli numbers. Moreover, we also prove that $C_{i}(q)\in \mathbb{Q}[E_2,E_4,E_6]$, where $E_2$, $E_4$ and $E_6$ are the classical Eisenstein series of weight 2, 4 and 6, respectively.
\end{abstract}

\section{Introduction}
Let $0<q<1$. We consider the functional differential equation
\begin{align}\label{FDE}
y'(x) = y(qx).
\end{align}
If $y(0)=1$, then \eqref{FDE} has the unique solution
\begin{equation}\label{fx}
y=f(x) := \sum\limits_{n = 0}^\infty  {\frac{{{x^n}}}{{n!}}{q ^{ n(n - 1)/2}}}.
\end{equation}
The function $f(x)$ is called the deformed exponential function since when $q\rightarrow 1$, $f(x)\to e^{x}$.
It appears naturally and frequently in pure mathematics as well as statistical physics. For instance, the function $f(x)$ relates closely to the generating function for Tutte polynomials of the complete graph $K_n$ in combinatorics, the Whittaker and Goncharov constants in complex analysis, and the partition function of one-site lattice gas with fugacity $x$ and two-particle Boltzmann weight $q$ in statistical mechanics \cite{scott2005}.

Among all the mysterious properties of $f(x)$, people are extremely interested in the zeros of $f(x)$. In 1972, Morris et al. \cite{morris} showed that $f(x)$ is an entire function of order zero. Moreover, by using  a theorem of Laguerre, they proved that $f(x)$ has infinitely many real zeros and these zeros are all negative and simple. They also proved that there is no other zero for the analytic extension (to the complex plane) of $f(x)$ by using the so-called multiplier sequence (a modest gap in their proof was filled by Iserles \cite{iserles}).
Therefore, the zeros of $f(x)$ form one strictly decreasing sequence of negative numbers $(x_k)$ ($k\ge 1$). We remark that in some previous works (e.g., \cite{langley,sokal2012}), the subscripts of the sequence start with 0 rather than 1. In this paper, as well as in \cite{Zhang}, the subscripts start with 1 for the elegance of notation.

Some conjectures on the zeros $x_k$ ($k\ge 1$) have been proposed in \cite{morris,iserles,Robinson}. For example, Morris et al. \cite{morris} conjectured that
\begin{align}\label{Morris-conj}
\lim\limits_{k\to \infty}\frac{x_{k+1}}{x_k}=\frac{1}{q}.
\end{align}
In 1973, Robinson \cite{Robinson} also derived \eqref{FDE} when counting the labeled acyclic digraphs. He speculated that
\begin{align}
x_{k}=-kq^{1-k}+o(q^{1-k}).
\end{align}
These conjectures have been investigated by several authors (see e.g., \cite{Grabner,langley,liu,sokal2012}). In particular,  Langley \cite{langley} showed that as $k\to\infty$
\begin{equation}\label{ratio}
\frac{x_{k+1}}{x_{k}}=\frac1q\Big(1+\frac1k\Big)+o(k^{-2}).
\end{equation}
He also proved that there exists a positive constant $\gamma$, which is independent of $k$, such that
\begin{equation}\label{xnlangley}
x_{k}=-kq^{1-k}(\gamma+o(1)).
\end{equation}
As a consequence, \eqref{Morris-conj} is true. Around 2009, more interesting conjectures on the zeros were introduced by Sokal \cite{sokal2009} by allowing $q$ to be in the unit disk of the complex plane.

Recently, Zhang \cite{Zhang} refined Langley's work and proved that the constant in \eqref{xnlangley} is $\gamma=1$. Moreover, he improved \eqref{xnlangley} and showed that as $k\to\infty$,
\begin{align}\label{Zhang-eq}
x_k=-kq^{1-k}\Big(1+\sum_{m=1}^\infty\sigma(m)q^mk^{-2}+o(k^{-2})\Big).
\end{align}
Here for any positive integer $n$,
\begin{align*}
\sigma(n):=\sum_{d|n,\ d>0}d.
\end{align*}
 Later Derfel et al. \cite{DGT} studied the asymptotic behaviours of the zeros of solutions of \eqref{FDE} with different initial conditions instead of the restriction $y(0)=1$.

In this paper, we find a complete asymptotic expansion formula for $x_{k}$. To be more specific, we will approximate $x_{k}$ with remainder term $o(k^{-n-1})$ for any $n\ge 1$. This extends \eqref{Zhang-eq} to the most general situation. To state our result, we define for $i\ge 0$,
\[A_i=A_i(q):=\sum_{m=1}^\infty m^i\sigma(m)q^m.\]
\begin{theorem}\label{thm1}
Let $n\ge1$. Then  as $k\to \infty$,
\begin{align}\label{xk-thm}
x_k=-kq^{1-k}\Big(1+\sum_{i=1}^{n}C_i(q)k^{-1-i}+o(k^{-1-n})\Big),
\end{align}
where each $C_i(q)$ is a multivariate polynomial of $A_0,A_1,...,A_{i-1}$ with rational coefficients. This polynomial can be determined recursively.
\end{theorem}
The recursive relation and the structure of these polynomials will be presented in Sections \ref{proof} and \ref{representation}. For example, $C_1=A_0$, $C_2=-A_1$, $C_3=-\frac1{10}A_0+\frac35A_1+\frac12A_2-\frac{13}{10}A_0^2$. It is clear that \eqref{Zhang-eq} is a special case of Theorem \ref{thm1}.

When $n\ge 4$, we observe that the expression of $C_{n}(q)$ in terms of $A_0,A_1,\dots, A_{n-1}$ is not unique. The polynomial given by the recursive relation in Theorem \ref{thm1} is just one candidate. For example, we have
\begin{align}
C_{4}&=\frac{1}{10}A_1-\frac{14}{15}A_2-\frac{1}{6}A_3+\frac{23}{5}A_{0}A_{1} \nonumber \\
&=\frac{1}{10}A_1-\frac{11}{10}A_2+\frac{23}{5}A_0A_1-6A_{1}^2+4A_{0}A_{2}.
\end{align}
Thus we continue to study the relations between $A_{i}$'s. Indeed, the following identity is established
\[A_3=A_2+36A_1^2-24A_0A_2.\]
Differentiating it gives more similar identities on $A_i$'s. Therefore, we find that it is possible to express $C_{n}$ as a polynomial in just $A_0, A_1$ and $A_2$. Furthermore, the coefficients of the linear terms in that polynomial can be given explicitly using Bernoulli numbers. Let $B_{n}$ be the $n$-th Bernoulli number. It is well known that $B_{2m+1}=0$ for all $m\ge 1$. The first few values of $B_{i}$ are $B_{0}=1$, $B_{1}=-\frac{1}{2}$, $B_{2}=\frac{1}{6}$ and $B_{4}=-\frac{1}{30}$.
\begin{theorem}\label{thm2}
For any $n\ge 1$, $C_{n}$ can be expressed as a polynomial in $A_0, A_1$ and $A_2$ with rational coefficients. This polynomial is unique and for $n\ge 2$, we have
\begin{align*}
C_{2n-1}&=\frac{6B_{2n}}{n}A_0-\frac{36B_{2n}}{n}A_1+\left(1+\frac{30B_{2n}}{n} \right)A_2+{\rm higher \ degree \ terms}, \\
C_{2n}&=-\frac{6B_{2n}}{n}A_1+\left(\frac{6B_{2n}}{n}-1 \right)A_2+ {\rm higher\ degree\ terms}.
\end{align*}
\end{theorem}
For example, we find that
\begin{align*}
C_{5}=&\frac{1}{21}A_0-\frac{2}{7}A_1+\frac{26}{21}A_2+\frac{53}{70}A_{0}^2+22A_{1}^2-36A_0A_{1}^2\nonumber \\
&-\frac{159}{35}A_0A_1-\frac{43}{2}A_0A_2+2A_1A_2+\frac{737}{210}A_{0}^3+24A_{0}^2A_{2}, \\
C_{6}=&-\frac{1}{21}A_1-\frac{20}{21}A_2-\frac{74}{35}A_0A_1-\frac{1401}{35}A_{1}^2-\frac{2}{5}A_{2}^2+\frac{705}{14}A_{0}A_{2}\\
&-\frac{101}{10}A_1A_2+\frac{1662}{5}A_{0}A_{1}^2-\frac{321}{14}A_{0}^2A_{1}-\frac{36}{5}A_{1}^3\\
&-\frac{1132}{5}A_{0}^2A_{2}-\frac{864}{5}A_{0}^2A_{1}^2+\frac{72}{5}A_0A_1A_2+\frac{576}{5}A_{0}^3A_{2}.
\end{align*}
Let
\begin{align}
E_2&=E_2(q):=1-24\sum_{n=1}^{\infty}\frac{nq^n}{1-q^n}, \label{P-defn}\\
E_4&=E_4(q):=1+240\sum_{n=1}^{\infty}\frac{n^3q^n}{1-q^n}, \label{Q-defn}\\
E_6&=E_6(q):=1-504\sum_{n=1}^{\infty}\frac{n^5q^n}{1-q^n}. \label{R-defn}
\end{align}
It is well known that $E_2, E_4$ and $E_6$ are classical Eisenstein series on the full modular group
\begin{align*}
SL(2,\mathbb{Z})=\left\{\begin{pmatrix} a & b \\ c& d\end{pmatrix}\mid ad-bc=1,\ a,b,c,d \in \mathbb{Z}  \right\}.
\end{align*}
We will show that $A_0$, $A_1$ and $A_2$ can be represented as polynomials in $E_2, E_4$ and $E_6$ with rational coefficients and vice versa. This in turn implies that $C_{n}(q)\in \mathbb{Q}[E_2,E_4,E_6]$, a subset of the ring of quasimodular forms on $SL(2,\mathbb{Z})$ \cite{Martin}. Since it is well known that $E_2$, $E_4$ and $E_6$ are algebraically independent over $\mathbb{C}$, it follows that $A_0$, $A_1$ and $A_2$ are also algebraically independent over $\mathbb{C}$.

The paper is organized as follows. In Section \ref{pre} we collect some results that is necessary for proving Theorem \ref{thm1}. Specifically,  we analyze the values of $f(-(k+ak^{-1})q^{1-k})$. The crux of the analysis is the series expansion of
\begin{align}\label{key-series}
\frac{G(x)-H(x)}{x^2}(1+a(x)x^2)
\end{align}
where
\begin{align}
G(x)&:=\prod_{i=0}^{j-1}\frac{1+ax^2}{1+ix}, \label{G-defn} \\
H(x)&:=\prod_{i=1-j}^{-1}\frac{1+ix}{1+ax^2} \label{H-defn}
\end{align}
and
\begin{align*}
a(x)=\sum_{n=0}^{\infty}a_{n}x^n.
\end{align*}
In Section \ref{proof}, we first define $C_{n}$ recursively by exploiting the coefficients of the series expansion of \eqref{key-series}. Then we determine the signs of $f(x)$ at the endpoints of certain intervals. We finish the proof of Theorem \ref{thm1} by the Intermediate Value Theorem. Section \ref{representation} is devoted to the representations of $C_{n}$. We give more details on the recursive formula of $C_n$, and discuss the structure of the multivariate polynomial representations of $C_n$ by those $A_i$'s, especially the linear terms. Then we establish a relation between the classical Eisenstein series and our $A_i$'s, and complete the proof of Theorem \ref{thm2}.
\begin{remark}\label{remark1}
Suppose that $C_i(q)=\sum_{j=1}^\infty C_{ij}q^j$. Using formal power series, we denote $F_j(x):=\sum_{i=1}^\infty C_{ij}x^{i+1}$. Then one may rewrite the asymptotic expansion as a formal power series
\begin{align}\label{add-xk-expan}
x_k(q)\sim -kq^{1-k}\Big(1+\sum_{j=1}^\infty F_j(k^{-1})q^j\Big).
\end{align}
We observe that the formal power series \eqref{add-xk-expan} numerically agrees with the expansion in $q$ of the $k$-th zero given in \cite[p.14]{sokal2009}. Note that in \cite{sokal2009} the sequence $(x_k)$ starts with subscript 0 and for each $j$, $F_j(k^{-1})$ is a rational function in $k$. It was conjectured by Sokal \cite[p.11]{sokal2009} that $F_j(k^{-1})\ge0$ for all integers $j\ge1$, $k\ge1$. In particular, we will see that $C_{i1}=(-1)^{i+1}$ by Proposition \ref{sumcoeff}, which implies that \[F_1(k^{-1})=\frac{1}{k(k+1)}.\]However, the difficulty to determine the closed form of $F_j(k^{-1})$ for $j\ge2$ lies in the complexity of those nonlinear terms in the polynomials $C_i=C_i(A_0,...,A_{i-1})$, $i\ge1$.
\end{remark}

\section{Preliminary Results}\label{pre}
We study the values of $f(-(k+ak^{-1})q^{1-k})$ for large $k$, where $a=a(k^{-1})$ is a function in $k$ which is positive and bounded.
We first observe that
\begin{align}\label{f-exp}
f( - (k + a {k^{ - 1}}){q^{1-k}})=  \sum\limits_{n = 0}^\infty  {{( - 1)}^n{u_n}},
\end{align}
where
\[u_n={\frac{{{{(k + a {k^{ - 1}})}^n}}}
{{n!}}} {q^{-n(2k - n- 1)/2}}.\]
We denote
\[{v_j} = {u_{2k - j - 1}} - {u_j}\quad (0 \le j \le k - 1).\]
In the following lemma, we use the notation ``$O(k^{-m})$'' to denote that the remainder is bounded by $Ck^{-m}$, where $C$ is an absolute constant independent of $N$ and $k$. This lemma shows the positivity and ``almost monotonicity'' of the sequence $v_j$.
\begin{lemma}\label{lemmavj}
For any integer $k \ge 1$, we have
\[\quad {v_j} > 0, \quad \ 0 \le j \le k - 1.\]
Furthermore, suppose that $a=a_{0}+O(k^{-1})$ for some constant $a_{0}>0$, then there exists a positive integer $N(q)$ such that for any $N\ge N(q)$ and $k\ge q^{-3N}$,
\[{v_j} < {v_{j + 1}}, \quad \ 0 \le j \le k - N.\]
\end{lemma}
\begin{proof}
Note that by the AM-GM inequality, we have
\begin{align}\label{AM-start}
\prod\limits_{i = 1}^{2k - 1 - 2j} {(j + i)} & <\left(\frac{\sum\limits_{i=1}^{2k-1-2j}(j+i)}{2k-1-2j} \right)^{2k-1-2j}\nonumber \\
&= {k^{2k - 1 - 2j}} \nonumber \\
&< {(k + a{k^{ - 1}})^{2k - 1 - 2j}}\quad (0 \le j \le k - 1).
\end{align}
This implies
\[\frac{{{{(k + a{k^{ - 1}})}^j}}}
{{j!}} < \frac{{{{(k + a{k^{ - 1}})}^{2k - 1 - j}}}}
{{(2k - 1 - j)!}}\quad (0 \le j \le k - 1).\]
So
\[{\rm{ }}{u_j} < {u_{2k - 1 - j}}\quad (0 \le j \le k - 1),\]
which gives the first inequality.

Note that
\begin{align}\label{vj-exp}
v_{j+1}=q^{(j+1)(j+2-2k)/2}\frac{(k+ak^{-1})^{2k-j-2}}{(2k-2-j)!}(1-w_{j})
\end{align}
where
 \[{w_j} = \frac{{(2k - 2 - j)!}}
{{(j + 1)!{{(k + a{k^{ - 1}})}^{2k - 3 - 2j}}}}.\]
Since we have proved that $v_{j+1}>0$, it follows that
\[0<w_j<1\]
By the AM-GM inequality,
\[\frac{{{w_{j + 1}}}}
{{{w_j}}} = \frac{{{{(k + a {k^{ - 1}})}^2}}}
{{(j + 2)(2k - 2 - j)}} > 1.\]

By \eqref{vj-exp}, we see that $v_{j}<v_{j+1}$ is equivalent to
\begin{align}\label{v-incre-1}
q^{j+1-k}(1-w_{j})>\frac{k+ak^{-1}}{2k-j-1}(1-w_{j-1}).
\end{align}
Using the relation
\begin{align*}
w_{j-1}=\frac{(j+1)(2k-j-1)}{(k+ak^{-1})^2},
\end{align*}
we see that \eqref{v-incre-1} is equivalent to
\begin{align}\label{v-incre-2}
\left( {{q ^{-k + j + 1}} - \frac{{j + 1}}
{{k + a {k^{ - 1}}}}} \right){w_j} < {q^{-k + j + 1}} - \frac{{k + a {k^{ - 1}}}}
{{2k - 1 - j}}.
\end{align}

For some positive integer $N$, we denote
\[t = 2k - 1 - j\quad (k + N-1 \le t \le 2k - 1)\]
and
\[g(t) = \frac{1}
{t} - \frac{{(2k - t){w_{k - N}}}}
{{{{(k + a {k^{ - 1}})}^2}}} - {q ^{k - t}}\left( {\frac{{1 - {w_{k - N}}}}
{{k + a {k^{ - 1}}}}} \right).\]
Direct calculation yields
\begin{align}\label{g-1st}
g'(t) =  - \frac{1}
{{{t^2}}} + \frac{{{w_{k - N}}}}
{{{{(k + a {k^{ - 1}})}^2}}} + \left( {\frac{{1 - {w_{k - N}}}}
{{k + a {k^{ - 1}}}}} \right){q ^{k -t}}\ln q
\end{align}
and
\begin{align}\label{g-2nd}
g''(t) = \frac{2}
{{{t^3}}} - \left( {\frac{{1 - {w_{k - N}}}}
{{k + a {k^{ - 1}}}}} \right){q ^{k - t}}{(\ln q )^2}.
\end{align}
Since $0<w_j<1$, $g''(t)$ is decreasing for $t>0$.

Note that when $N$ is large enough ($N\ge N_{1}(q)$), we have  $k\ge q^{-3N}\ge  N^{2}$. Hence
\begin{align}\label{estimate-1}
\frac{1}{(k+N-1)^3}=k^{-3}\left(1-3(N-1){k^{-1}}+O\left(N^2k^{-2}\right) \right)=k^{-3}+O(k^{-7/2}).
\end{align}
Next,
\begin{align*}
w_{k-N}&=\prod\limits_{t=-N+2}^{N-2}\frac{k+t}{k+ak^{-1}} \nonumber \\
&=\prod\limits_{t=-N+2}^{N-2}\left(1+\frac{t}{k}\right) \left(1-\frac{a_0}{k^2}+O(k^{-3}) \right) \nonumber \\
&=\prod\limits_{t=-N+2}^{N-2}\left(1+\frac{t}{k}-\frac{a_0}{k^2}-\frac{a_0t}{k^3}+O(k^{-3}) \right) \nonumber\\
&=1-a_0(2N-3)k^{-2}+\sum\limits_{-N+2\le t_1< t_2\le N-2}\frac{t_1t_2}{k^2}+O(k^{-5/2}) \nonumber \\
&=1-a_0(2N-3)k^{-2}-\frac{(N-2)(N-1)(2N-3)}{6}k^{-2}+O(k^{-5/2}).
\end{align*}
Hence
\begin{align}\label{estimate-2}
\frac{1-w_{k-N}}{k+ak^{-1}}=&\frac{1}{k}\left(1-a_0k^{-2}+O(k^{-3}) \right) \nonumber \\
&\cdot \left(a_0(2N-3)k^{-2}+\frac{(N-2)(N-1)(2N-3)}{6}k^{-2}+O(k^{-5/2})  \right) \nonumber \\
=&\left(a_0(2N-3)+\frac{(N-2)(N-1)(2N-3)}{6} \right)k^{-3}+O(k^{-7/2}).
\end{align}
Note that $k\ge q^{-3N}$ implies $q^{-N}=O(k^{1/3})$. Now by \eqref{estimate-1} and \eqref{estimate-2}, we deduce that
\begin{align*}
g''(k+N-1)=\left( {2 - c_{N}{q ^{1 - N}}{{(\ln q )}^2}} \right){k^{ - 3}} + O({k^{ - 19/6}})
\end{align*}
where
\[c_{N} := a_0(2N - 3) + \frac{1}{6}(N - 2)(N - 1)(2N - 3).\]
Similarly, we have
\begin{align*}
g'(k + N - 1) = \left( {2N - 2 + c_{N}{q^{1 - N}}\ln q } \right){k^{ - 3}} + O({k^{ - 19/6}}).
\end{align*}
When $N$ is sufficiently large ($N\ge N_2(q)\ge N_1(q)$), we will we have
\[g''(k+N-1) <0,\]
and
\[g'(k+N-1) < 0.\]
So $g'(t)$ and $g(t)$ are also decreasing for $t\ge k+N-1$.

In the same way, we find that
\begin{align*}
g(k + N - 1) =& \left( {a_0 (2N - 1) + \frac{N}
{6}(N - 1)(2N - 1) - c_{N}{q ^{1 - N}}} \right){k^{ - 3}} \nonumber \\
&+ O({k^{ - 19/6}}).
\end{align*}
When $N$ is large enough ($N\ge N(q)\ge N_2(q)$), we have
\[g(k+N-1)<0.\]
So if $N\ge N(q)$ and $k\ge q^{-3N}$, we have
\[g(t) \le g(k + N-1) < 0, \quad t\ge k+N-1.\]
Therefore,
 \[\frac{1}
{{2k - 1 - j}} - \frac{{(j + 1){w_{k - N}}}}
{{{{(k + a {k^{ - 1}})}^2}}} - {q ^{-k + j + 1}}\left( {\frac{{1 - {w_{k - N}}}}
{{k + a {k^{ - 1}}}}} \right) < 0\quad (0 \le j \le k - N).\]
This implies
\[\left( {{q ^{-k + j + 1}} - \frac{{j + 1}}
{{k + a {k^{ - 1}}}}} \right){w_{k - N}} < {q ^{-k + j + 1}} - \frac{{k + a{k^{ - 1}}}}
{{2k - 1 - j}}\quad (0 \le j \le k - N).\]
Since $w_j<w_{j+1}$, we have
\[\left( {{q ^{-k + j + 1}} - \frac{{j + 1}}
{{k + a {k^{ - 1}}}}} \right){w_j} < {q^{-k + j + 1}} - \frac{{k + a {k^{ - 1}}}}
{{2k - 1 - j}}\quad (0 \le j \le k - N).\]
This proves \eqref{v-incre-2} and hence the fact that
\[{v_j} < {v_{j + 1}}\quad (0 \le j \le k - N).\]
\end{proof}

However, the sequence $v_j$ may not be monotone when $k-N<j\le k-1$. So we need more delicate analysis on these $v_j$'s, which is the crux of the problem. For $1\le j\le N$,
\begin{align}
v_{k-j}&=u_{k+j-1}-u_{k-j}=\Big(\frac{(k+ak^{-1})^{k+j-1}}{(k+j-1)!}-\frac{(k+ak^{-1})^{k-j}}{(k-j)!}\Big)q^{-(k+j-1)(k-j)/2} \nonumber\\
&=\left(\prod_{i=1}^{j-1}\frac{k+ak^{-1}}{k+i}-\prod_{i=1-j}^0\frac{k+i}{k+ak^{-1}}\right)(k+ak^{-1})^k\frac1{k!}q^{-(k+j-1)(k-j)/2} \nonumber \\
&=\left(\prod_{i=1}^{j-1}\frac{1+ax^2}{1+ix}-\prod_{i=1-j}^0\frac{1+ix}{1+ax^2}\right)\frac{(1+ax^2)^2}{x^2}q^{j(j-1)/2}\cdot(k+ak^{-1})^{k-2}\frac1{k!}q^{-k(k-1)/2} \nonumber\\
&=(G(x)-H(x))\frac{1+ax^2}{x^2}q^{j(j-1)/2}\cdot(k+ak^{-1})^{k-2}\frac1{k!}q^{-k(k-1)/2}, \label{vkj}
\end{align}
where $G(x)$ and $H(x)$ was defined in \eqref{G-defn}--\eqref{H-defn} and here we set $x=k^{-1}$. In particular, $G(x)=1+ax^2$ and $H(x)=1$ when $j=1$. Now we describe the series expansion of \eqref{key-series}.
\begin{lemma}\label{coeff}
Let
\begin{align}
a=a(x):=\sum_{i=0}^{\infty}a_{i}x^{i}.
\end{align}
For $n\ge 1$, the coefficient of $x^{n-1}$ in the expansion of
\[\frac{G(x)-H(x)}{x^2}(1+a(x)x^2)\]
has the form
\begin{equation}\label{goal}
u(a_{n-1}+S_{0}(n)+S_{1}(n)v+S_{2}(n)v^2+...+S_{n}(n)v^n),
\end{equation}
where $u=2j-1$, $v=j(j-1)$, and each $S_{i}(n)$ is a polynomial of $a_0,a_1,...,a_{n-2}$, which is independent of $j$ and has rational coefficients. In particular, $S_{0}(1)=0$ and $S_{1}(1)=\frac{1}{6}$.
\end{lemma}
The rest of this section will be devoted to giving a proof of Lemma \ref{coeff}.  First, we compute the coefficients in the expansions $G(x)$ and $H(x)$
\begin{align*}
G(x)=\sum_{N=0}^{\infty}G_Nx^N, \quad H(x)=\sum_{N=0}^\infty H_Nx^N,
\end{align*}
where  $G_N$ and $H_N$ are polynomials of $a$ with coefficients depending on $j$.
For example,
\begin{align*}
G_0=&H_0=1, \\
G_1=&H_1=-\tfrac12j(j-1),\\
G_2=&\tfrac1{24}(j-1)j(j+1)(3j-2)+ja,\\
H_2=&\tfrac1{24}(j-2)(j-1)j(3j-1)+(1-j)a,\\
G_3=&-\tfrac1{48}(j-1)^2j^2(j+1)(j+2)-\tfrac12(j-1)j^2a,\\
H_3=&-\tfrac1{48}(j-3)(j-2)(j-1)^2j^2+\tfrac12(j-1)^2ja.
\end{align*}
To represent $G_N$ and $H_N$, we define for $j\ge 1$,
\begin{align}\label{sigma-rec}
\sigma_0&=\sigma_0(j):=1,\nonumber \\
\sigma_i&=\sigma_{i}(j):=\sum_{1\le n_1<...<n_i\le j-1}n_1n_2\cdot\cdot\cdot n_i,\ \ 1\le i\le j-1,
\end{align}
and $\sigma_i(j)=0$ if $i\ge j$. Here we note that $\sigma_{i}(j)$ is just the unsigned Stirling numbers of the first kind
\begin{align*}
\sigma_{i}(j)=c(j,j-i).
\end{align*}
Similarly, we define for $j\ge1$,
\begin{align}\label{Q-rec}
Q_0=Q_0(j):=1,\ \ Q_k=Q_k(j):=-\sum_{i=1}^{k}\sigma_iQ_{k-i},\ k\ge1.
\end{align}
Recall the generalized binomial coefficient for $k\in\mathbb{N}$, $\alpha\in \mathbb{C}$
\[{{\alpha}\choose{k}}=\frac{\alpha(\alpha-1)\cdots (\alpha-k+1)}{k!}.\]
\begin{lemma}\label{GNHN}
We have
\[G_N=\sum_{m=0}^{[N/2]}G(N,m)a^m,\ \ N=0,1,2,...\]
where
\[G(N,m)={{j}\choose{m}}Q_{N-2m}(j),\]
Moreover,
\[H_N=\sum_{m=0}^{[N/2]}H(N,m)a^m,\ \ N=0,1,2,...\]
where
\[H(N,m)=(-1)^N{{1-j}\choose{m}}\sigma_{N-2m}(j).\]
\end{lemma}
\begin{proof} Since
\begin{align}\label{Q-gen}
\prod_{i=0}^{j-1}\frac1{1+ix}=\frac1{\sum_{i=0}^\infty{\sigma_ix^i}}=\sum_{i=0}^\infty Q_kx^k,
\end{align}
we have
\begin{align*}
 G(x)=\sum_{m=0}^j{{j}\choose{m}}a^mx^{2m}\sum_{k=0}^\infty Q_kx^k=\sum_{N=0}^\infty x^N\sum_{m=0}^{[N/2]}{{j}\choose{m}}Q_{N-2m}a^m.
\end{align*}
Similarly,
\begin{align*}
H(x)=\sum_{m=0}^\infty{{1-j}\choose{m}}a^mx^{2m}\sum_{k=0}^\infty \sigma_k(-x)^k=\sum_{N=0}^\infty x^N\sum_{m=0}^{[N/2]}(-1)^N\sigma_{N-2m}{{1-j}\choose{m}}a^m.
\end{align*}
\end{proof}

For fixed $n$, both $\sigma_n(j)$ and $Q_n(j)$ are polynomials in $j$. Hence they can be naturally extended to be two functions defined on the whole real line. The following lemma gives a relation between these two functions.
\begin{lemma} \label{replace}
Let $n\ge0$. Then we have $Q_n(1-t)=(-1)^n\sigma_n(t)$, $t\in \mathbb{R}$.
\end{lemma}
\begin{proof}
We denote for $j\ge1$,
\begin{align}
\phi_{j}(x)&:=\prod\limits_{i=0}^{j-1}\frac{1}{1+ix}=\sum_{n=0}^{\infty}Q_{n}(j)x^n, \label{A-def} \\
\psi_{j}(x)&:=\prod\limits_{i=1}^{j-1}(1-ix)=\sum_{n=0}^{\infty}(-1)^n\sigma_{n}(j)x^n. \label{B-def}
\end{align}
In particular, $\phi_1(x)=\psi_1(x)=1$.
Note that
\begin{align}
\phi_{j}(x)=(1+jx)\phi_{j+1}(x)=\sum_{n=0}^{\infty}\big(Q_{n}(j+1)+jQ_{n-1}(j+1) \big)x^n.
\end{align}
Comparing the coefficient of $x^n$ on both sides, we deduce that
\begin{align}\label{a-rec-1}
Q_{n}(j)=Q_{n}(j+1)+jQ_{n-1}(j+1).
\end{align}
Similarly, observing that $\psi_{j+1}(x)=\psi_{j}(x)(1-jx)$, we deduce that
\begin{align}\label{b-rec-1}
\sigma_{n}(j+1)=\sigma_{n}(j)+j\sigma_{n-1}(j).
\end{align}
Since \eqref{a-rec-1} and \eqref{b-rec-1} hold for all $j\ge1$ and both $Q_{n}(t)$ and $\sigma_{n}(t)$ are polynomials in $t$, we conclude that for any $t\in \mathbb{R}$\begin{align}\label{a-poly-1}
Q_{n}(t)=Q_{n}(t+1)+tQ_{n-1}(t+1),
\end{align}
\begin{align}\label{b-poly-1}
\sigma_{n}(t+1)=\sigma_{n}(t)+t\sigma_{n-1}(t).
\end{align}
Now we let $\overline{Q}_{n}(t)=Q_{n}(1-t)$. Then \eqref{a-poly-1} implies
\begin{align}\label{over-a}
\overline{Q}_{n}(t+1)=\overline{Q}_{n}(t)-t\overline{Q}_{n-1}(t).
\end{align}
Comparing \eqref{over-a} with \eqref{b-poly-1}, we see that the polynomials $\overline{Q}_{n}(t)$ and $(-1)^n$$\sigma_{n}(t)$ satisfy the same recurrence relation. Next, by direct computation, we find that
\begin{align}
Q_{0}(j)=\sigma_{0}(j)=1, \quad Q_{1}(j)=-\frac{j(j-1)}{2}, \quad  \sigma_{1}(j)=\frac{j(j-1)}{2}.
\end{align}
Thus
\begin{align*}
\overline{Q}_{1}(t)=-\sigma_{1}(t)=-\frac{t(t-1)}{2}.
\end{align*}
Now suppose that $\overline{Q}_{n-1}(t)=(-1)^{n-1}\sigma_{n-1}(t)$ for some $n\ge 2$. By \eqref{b-poly-1} and \eqref{over-a} we deduce that
\begin{align*}
\overline{Q}_{n}(t+1)-\overline{Q}_{n}(t)=(-1)^n{\sigma}_{n}(t+1)-(-1)^n\sigma_{n}(t).
\end{align*}
Summing over $t$ from 1 to $j-1$, we obtain
\begin{align}\label{over-a-b}
\overline{Q}_{n}(j)-\overline{Q}_{n}(1)=(-1)^n\sigma_{n}(j)-(-1)^n\sigma_{n}(1).
\end{align}
By definition, we have $Q_{n}(1)=(-1)^n\sigma_{n}(1)=0$ for $n\ge 1$. Therefore, \eqref{over-a-b} implies that $\overline{Q}_{n}(j)=(-1)^n\sigma_{n}(j)$ for any $j\geq 1$. This implies that $\overline{Q}_{n}(t)=(-1)^n\sigma_{n}(t)$ for any $t\in \mathbb{R}$.
\end{proof}

Since $G_0=H_0$, $G_1=H_1$, we get
\[\frac{G(x)-H(x)}{x^2}=\sum_{N=0}^\infty (G_{N+2}-H_{N+2})x^N.\]
Hence
\begin{align}\label{main}
\frac{G(x)-H(x)}{x^2}(1+a(x)x^2)=\sum_{N=0}^\infty(G_{N+2}-H_{N+2}+(G_N-H_N)a(x))x^N.
\end{align}
To prove Lemma \ref{coeff}, we need to compute the coefficient of $x^{n-1}$ in \eqref{main}. For $N\ge 2m$, we define
\begin{align}\label{Delta-defn}
\Delta(N,m):=G(N,m)-H(N,m)={{j}\choose{m}}Q_{N-2m}(j)-(-1)^N{{1-j}\choose{m}}\sigma_{N-2m}(j).
\end{align}
For example,
 \[\Delta(0,0)=\Delta(1,0)=0,\]
 \[\Delta(2,0)=\tfrac16j(j-1)(2j-1),\ \ \ \Delta(2,1)=2j-1,\]
 \[\Delta(3,0)=-\tfrac1{12}(j-1)^2j^2(2j-1),\ \ \ \Delta(3,1)=-\tfrac12(j-1)j(2j-1),\]
 \[\Delta(4,0)=\tfrac1{240}(-1 + j) j (-1 + 2 j) (-4 - 12 j + 17 j^2 - 10 j^3 + 5 j^4),\]
 \[\Delta(4,1)=\tfrac1{24} (-1 + j) j (-1 + 2 j) (2 - 3 j + 3 j^2),\ \ \ \Delta(4,2)=0,\]
 \[\Delta(5,0)=-\tfrac1{1440}(-1 + j)^2 j^2 (-1 + 2 j) (-12 - 56 j + 61 j^2 - 10 j^3 + 5 j^4),\]
 \[\Delta(5,1)=-\tfrac1{48}(-1 + j)^2 j^2 (-1 + 2 j) (6 - j + j^2),\ \ \ \Delta(5,2)=0.\]

\begin{proposition}\label{Delta-uv-exp}
Let $u=2j-1$ and $v=j(j-1)$. Then $\Delta(N,m)$ can be written as
\begin{align}
\Delta(N,m)=u(s_0+s_1v+\cdots +s_kv^k), \quad s_0, s_1, \dots,s_k \in \mathbb{Q},
\end{align}
where $N\ge2m$ and $k \le \left[\frac{2N-3m-1}{2} \right]$. Moreover, if $(N,m)\ne (2,1)$, then $s_0=0$.
\end{proposition}
In order to prove this proposition, we need the following lemmas.
\begin{lemma}\label{uv}
 Any polynomial of $j$ can be written into a polynomial of $u=2j-1$ and $v=j(j-1)$, and the degree of $u$ in each term is at most 1. In order words, for a polynomial $\varphi(j)$, we can write it as
\[\varphi(j)=s_{0,0}+s_{0,1}v+...+s_{0,n_0}v^{n_0}+u(s_{1,0}+s_{1,1}v+...+s_{1,n_1}v^{n_1}).\]
Moreover, if all the coefficients of $\varphi(j)$ are rational numbers, then each $s_{i,l}\in \mathbb{Q}$.
\end{lemma}
\begin{proof}
The assertions are clearly true if the polynomial has degree 1. If the polynomial has the form $j^2+aj+b$, then
\[j^2+aj+b=j(j-1)+\tfrac{a+1}2(2j-1)+b+\tfrac{a+1}2.\]
 Higher-degree cases can be done similarly.
\end{proof}
\begin{lemma}\label{sigmaQ}
For $k\ge1$,
\begin{align}
\sigma_k(j)&=\frac1{2^kk!}j^{2k}-\frac{2k+1}{3\cdot 2^k(k-1)!} j^{2k-1}+ O(j^{2k-2}),  \label{sigma-expansion} \\
Q_k(j)&=\frac{(-1)^k}{2^kk!}j^{2k}+\frac{(-1)^k(2k-5)}{3\cdot 2^k(k-1)!}j^{2k-1}+O(j^{2k-2}).\label{Q-expansion}
\end{align}
\end{lemma}
\begin{proof}
We denote
\begin{align*}
p_{m}=\sum_{k=1}^{j-1}k^m.
\end{align*}
It is known that
\begin{align}\label{p-m}
p_{m}=\frac{1}{m+1}\sum_{i=0}^{m}(-1)^i\binom{m+1}{i}B_{i}j^{m+1-i}-j^m.
\end{align}
 From \eqref{p-m} we have
\begin{align*}
p_{1}=\frac{1}{2}j^2-\frac{1}{2}j, \quad p_{2}=\frac{1}{3}j^3-\frac{1}{2}j^2+\frac{1}{6}j.
\end{align*}
As a polynomial in $j$, the degree of $p_{m}$ is $m+1$.
Moreover, it is known that
\begin{align}\label{sigma-p-exp}
\sigma_{k}=(-1)^k\sum_{\begin{smallmatrix} m_1+2m_2+\cdots +km_k=k \\ m_1\ge 0, \cdots, m_{k}\ge 0\end{smallmatrix}}\prod\limits_{i=1}^{k}\frac{(-p_{i})^{m_i}}{m_{i}!i^{m_i}}.
\end{align}
As a polynomial in $j$, the degree of $\sigma_{k}$ is no more than
\begin{align*}
\deg\big(\prod\limits_{i=1}^{k}p_{i}^{m_i} \big)&=2m_1+3m_2+\cdots +(k+1)m_k \nonumber \\
&=(m_1+2m_2+\cdots +km_k)+(m_1+m_2+\cdots +m_k) \nonumber \\
&=k+(m_1+m_2+\cdots +m_k) \nonumber \\
&\leq 2k.
\end{align*}
Now we are going to find the coefficients of $j^{2k}$ and $j^{2k-1}$ in $\sigma_{k}$, respectively.

We consider the system of linear equations
\begin{align*}
\left\{\begin{array}{l}
2m_1+3m_2+\cdots +(k+1)m_k=2k \\
m_1+2m_2+\cdots +km_k=k,
\end{array}\right.
\end{align*}
which is equivalent to
\begin{align*}
\left\{\begin{array}{l}
m_1+m_2+\cdots +m_k=k \\
m_1+2m_2+\cdots +km_k=k.
\end{array}\right.
\end{align*}
It is clear that the unique solution to the equations above are $m_1=k$ and $m_i=0$ for $2\le i \le k$. Now we compute
\begin{align}
(-1)^k\frac{(-p_1)^k}{k!}&=\frac{1}{k!}\left(\frac{1}{2}j^2-\frac{1}{2}j \right)^k \nonumber \\
&=\frac{1}{k!}\left(\frac{1}{2^k}j^{2k}-\frac{k}{2^{k}}j^{2k-1}+O(j^{2k-2})\right). \label{high-1}
\end{align}
Similarly, we consider the system of linear equations
\begin{align*}
\left\{\begin{array}{l}
2m_1+3m_2+\cdots +(k+1)m_k=2k-1 \\
m_1+2m_2+\cdots +km_k=k.
\end{array}\right.
\end{align*}
The only solutions are $m_1=k-2$, $m_2=1$ and $m_i=0$ for all $i \ge 3$. The corresponding term in $\sigma_{k}$ is
\begin{align}
(-1)^k\frac{(-p_1)^{k-2}}{(k-2)!}\cdot \frac{-p_{2}}{2!}&=-\frac{1}{2(k-2)!}\left(\frac{1}{2}j^2-\frac{1}{2}j\right)^{k-2}\cdot \left( \frac{1}{3}j^3-\frac{1}{2}j^2+\frac{1}{6}j\right) \nonumber  \\
&=-\frac{1}{3\cdot 2^{k-1}(k-2)!}j^{2k-1}+O(j^{2k-2}). \label{high-2}
\end{align}
Adding \eqref{high-1} and \eqref{high-2} up, we obtain \eqref{sigma-expansion}.

Next, from the proof of Lemma \ref{replace}, we see that if we replace $j$ by $1-j$ in the polynomial expression of $(-1)^k\sigma_{k}$, then we get $Q_{k}$. Thus by replacing $j$ by $1-j$ in \eqref{sigma-expansion}, we obtain \eqref{Q-expansion}.
\end{proof}
Now we are able to prove Proposition \ref{Delta-uv-exp}.
\begin{proof}[Proof of Proposition \ref{Delta-uv-exp}]
From \eqref{Delta-defn} it is clear that $\Delta(N,m)$ is a polynomial of $j$ with rational coefficients. By Lemma \ref{uv}, we can write
\begin{align}\label{Delta-start-1}
\Delta(N,m)=s_{0,0}+s_{0,1}v+\cdots+s_{0,n_0}v^{n_0}+u\left(s_{1,0}+s_{1,1}v+\cdots +s_{1,n_1}v^{n_1}\right)
\end{align}
where each $s_{i,l}\in \mathbb{Q}$.
Replacing $j$ by $1-j$, then $u\mapsto -u$ and $v\mapsto v$.
Lemmas \ref{GNHN} and \ref{replace}  imply
\begin{align}\label{Delta-start-2}
-\Delta(N,m)=s_{0,0}+s_{0,1}v+\cdots+s_{0,n_0}v^{n_0}-u\left(s_{1,0}+s_{1,1}v+\cdots +s_{1,n_1}v^{n_1}\right).
\end{align}
From \eqref{Delta-start-1} and \eqref{Delta-start-2}, we deduce that
\begin{align*}
\Delta(N,m)=u\left(s_{1,0}+s_{1,1}v+\cdots +s_{1,n_1}v^{n_1}\right).
\end{align*}
Moreover, it is well known that $p_{k}$ is divisible by $v$ for any $k\ge 1$. Hence from \eqref{sigma-p-exp} we know $\sigma_{k}$ is divisible by $v$ for any $k\ge 1$. From \eqref{Q-rec} we know $Q_{k}$ is divisible by $v$ for any $k\ge 1$. Hence \eqref{Delta-defn} implies that $\Delta(N,m)$ is divisible by $v$ when $N-2m>0$. If $N-2m=0$, we have
\[\Delta(2m,m)=\binom{j}{m}-\binom{1-j}{m}.\]
Clearly, when $j=0$ or $j=1$, we have $\Delta(2m,m)=0$ except when $m=1$. Therefore $\Delta(2m,m)$ has a factor $j(j-1)$ except when $m=1$. Thus when $(N,m)\ne (2,1)$, $\Delta(N,m)$ is always divisible by $v$, which means $s_{1,0}=0$.

It remains to prove that $n_1\le \left[\frac{2N-3m-1}{2}\right]$. By the definition of $G(N,m)$ and Lemma \ref{sigmaQ},
\begin{align}\label{GNm}G(N,m)=\frac{(-1)^{N}}{m!(N-2m)!2^{N-2m}}j^{2N-3m}+O(j^{2N-3m-1}).\end{align}
Similarly,
\begin{align}\label{HNm}H(N,m)=\frac{(-1)^{N-m}}{m!(N-2m)!2^{N-2m}}j^{2N-3m}+O(j^{2N-3m-1}).\end{align}
Hence
\begin{align}\label{deltaNm}
\Delta(N,m)={\rm Const.}\ uv^{[\frac{2N-3m-1}2]}+ {\rm lower\ degree\ terms}.
\end{align}
This completes the proof of Proposition \ref{Delta-uv-exp}.
\end{proof}

Finally, we arrive at the stage to prove Lemma \ref{coeff}.
\begin{proof}[Proof of Lemma \ref{coeff}]
We plug
\[a=\sum_{k=0}^{\infty}a_{k}x^k\]
into \eqref{main} and expand it to a power series of $x$.
By direct calculations, we find that the coefficient of $x^{n-1}$ in \eqref{main} is the sum of the following terms:
\begin{align}
\Delta(2,1)a_{n-1},\label{sum-term-1}
\end{align}
\begin{align}
&(\Delta(N,m)+\Delta(N-2,m-1))\sum_{i_1+...+i_m=n-N+1}a_{i_1}\cdot\cdot\cdot a_{i_m} \label{sum-term-3} \\
&(3\le N\le n+1,\ 1\le m\le [\tfrac N2]), \nonumber
\end{align}
\begin{align}
\Delta(n+1,0).\label{sum-term-4}
\end{align}
Note that $\Delta(2,1)=2j-1=u$. \eqref{sum-term-1} gives the first term in \eqref{goal}. By Proposition \ref{Delta-uv-exp}, we can write each $\Delta(N,m)$ as
\[u(s_0+s_1v+...+s_kv^k), \quad s_0, s_1, \dots, s_k \in \mathbb{Q},\]
where $k\le [\frac{2N-3m-1}2]$.  Since $N\le n+1$, $m\ge0$, we have $[\frac{2N-3m-1}2]\le n$, which means the degree of of $v$ in $\Delta(N,m)$ is at most $n$.  This gives \eqref{goal} and clearly each polynomial $S_{i}(n)$ has rational coefficients. More explicitly, from \eqref{sum-term-3} and \eqref{sum-term-4}, we see that for $i=0,1,...,n$,
\begin{equation}\label{sirepa}
\begin{aligned}
&S_i(n)=[uv^i]\Delta(n+1,0)\\
&+\sum_{N=3}^{n+1}\sum_{m=1}^{[N/2]}\sum_{i_1+...+i_m=n-N+1}([uv^i]\Delta(N,m)+[uv^i]\Delta(N-2,m-1))a_{i_1}\cdot\cdot\cdot a_{i_m},
\end{aligned}
\end{equation}
where $[uv^i]\Delta(N,m)$ means the coefficient of the term $uv^i$ in $\Delta(N,m)$. In particular, $S_i(1)=[uv^i]\Delta(2,0)$ for $i=0,1$. And $S_i(2)=[uv^i]\Delta(3,0)+[uv^i]\Delta(3,1)a_0$ for $i=0,1,2$. Recall that $\Delta(2,0)=\frac{1}{6}uv$, $\Delta(3,0)=-\frac1{12}uv^2$, and $\Delta(3,1)=-\frac12uv$. So we have $S_{0}(1)=S_0(2)=0$, $S_{1}(1)=\frac{1}{6}$, $S_1(2)=-\frac12 a_0$, $S_2(2)=-\frac1{12}$.
\end{proof}

\section{Proof of Theorem \ref{thm1}}\label{proof}
Recall Ramanujan's Theta-operator $\Theta=q\partial_q$, which has the effect that
\begin{align*}
\Theta\left(\sum_{n=n_0}^{\infty}a(n)q^n \right):=\sum_{n=n_0}^{\infty}na(n)q^n.
\end{align*}
Let
\[P_{0}=\sum_{j=1}^\infty(-1)^{j-1}(2j-1)q^{j(j-1)/2}.\]
\begin{lemma}\label{P-lemma}
For any $m\ge 1$, we have
\[\Theta^m(P_{0})=-3P_{0}P_m,\]
where $P_m$ is a multivariate polynomial of $A_0,A_1,...,A_{m-1}$ with rational coefficients. Moreover, we have
\[P_1=A_0,\ \ P_{m+1}=\Theta(P_m)-3A_0P_m,\ m=1,2,....\]
\end{lemma}
\begin{proof}
By Jacobi's identity \cite[Theorem 1.3.9]{Berndt}, we have
\begin{align}\label{Jacobi}
P_{0}=\prod\limits_{n=1}^{\infty}(1-q^n)^3.
\end{align}
Hence
\begin{align*}
\frac{\Theta(P_0)}{P_0}=-3\sum_{n=1}^{\infty}\frac{nq^{n}}{1-q^n}=-3\sum_{n=1}^{\infty}\sigma(n)q^n.
\end{align*}
This proves that $\Theta(P_0)=-3A_0P_0$ and hence $P_1=A_0$.

Next, note that
\begin{align*}
\Theta^{m+1}(P_0)=\Theta(-3P_0P_m)=-3\Theta(P_0)P_m-3P_0\Theta(P_m)=-3P_0(\Theta(P_m)-3A_0P_m).
\end{align*}
We deduce that
\begin{align}\label{P-rec}
P_{m+1}=\Theta(P_m)-3A_0P_m.
\end{align}

Suppose we have proved that $P_{m}$ is a polynomial of $A_0$, $A_1$, $\dots$, $A_{m-1}$ with rational coefficients, which is clear true for $m=1$. Then since $\Theta(A_k)=A_{k+1}$, from \eqref{P-rec} it follows that $P_{m+1}$ is a polynomial of $A_0$, $A_1$, $\dots$, $A_{m}$ with rational coefficients. Thus by induction on $m$ we know that the first assertion is true.
\end{proof}


Recall that in Lemma \ref{coeff}, $S_0(1)=0$, $S_{1}(1)=\frac{1}{6}$, and for $m \ge 2$, $S_{i}(m)$ are polynomials of $a_0$, $a_1$, $\cdots$, $a_{m-2}$ and independent of $j$. For $m\ge 1$, we recursively define
\begin{align}\label{rec}
C_{m}=-S_{0}(m)+\sum_{i=1}^{m}3\cdot 2^{i}S_{i}(m)(C_1,\cdots, C_{m-1})P_{i}.
\end{align}
In particular,
\begin{align}
C_{1}=6S_{1}(1)P_1=A_{0}.
\end{align}
The following lemma is a key for the proof of Theorem \ref{thm1}.
\begin{lemma}\label{sign}
Let $n\ge1$. Suppose $\lambda\ne C_{n}$. Then for large $k$,
\begin{align}\label{sign-eq}
(-1)^{k}(\lambda-C_{n})f(-(k+\Lambda_{n-1}(k^{-1})k^{-1})q^{1-k})>0,
\end{align}
where
\begin{align}\label{Lambda-defn}
\Lambda_{n-1}(x)=\sum_{i=1}^{n-1}C_{i}x^{i-1}+\lambda x^{n-1}.
\end{align}
\end{lemma}
\begin{remark}
Here and in the proof below, we use the convention that in any summation $\sum_{i=a}^{b}$, if $a>b$, then we assume the sum is empty (zero). In \eqref{Lambda-defn}, when $n=1$, we have an empty sum and so $\Lambda_{0}(x)=\lambda$.
\end{remark}
\begin{proof}
For convenience, we define for $m\ge 0$
\begin{align}
\overline{P}_{m}:=\sum_{j=1}^{\infty}(-1)^{j-1}(2j-1)\left(j(j-1) \right)^{m}q^{j(j-1)/2}
\end{align}
and
\begin{align}
\overline{P}_{m,2N-1}:=\sum_{j=1}^{2N-1}(-1)^{j-1}(2j-1)\left(j(j-1) \right)^{m}q^{j(j-1)/2}.
\end{align}
It is clear that $\overline{P}_0=P_0$ and $\overline{P}_{m}=2^{m}\Theta^{m}(P_{0})=-3\cdot 2^mP_0P_m$ ($m\ge 1$). Moreover, \eqref{rec} implies that for any $m\ge 1$,
\begin{align}\label{lambda-eq}
C_{m}P_0+\sum_{i=0}^{m}S_{i}(m)(C_1,C_2,\cdots, C_{m-1})\overline{P}_{i}=0.
\end{align}

Now we set $a=\Lambda_{n-1}(x)$ in \eqref{f-exp} with $x=k^{-1}$.
From \eqref{vkj} and Lemma \ref{coeff} we deduce that
\begin{align}
v_{k-j}=&(G(k^{-1})-H(k^{-1}))\frac{1+\Lambda_{n-1}(k^{-1})k^{-2}}{k^{-2}}q^{j(j-1)/2} \nonumber \\
&\cdot(k+\Lambda_{n-1}(k^{-1})k^{-1})^{k-2}\frac1{k!}q^{-k(k-1)/2} \nonumber \\
=&(k+\Lambda_{n-1}(k^{-1})k^{-1})^{k-2}\frac1{k!}q^{-k(k-1)/2}\cdot q^{j(j-1)/2}\Big(\sum_{m=1}^{n}\xi_{m-1}k^{-(m-1)}+O(k^{-n})  \Big), \label{v-new}
\end{align}
where
\begin{align}
\xi_{m-1}=u\left(C_{m}+\sum_{i=0}^{m}S_{i}(m)(C_1,\cdots, C_{m-1})v^{i}\right), \quad 1\leq m \leq n-1,
\end{align}
and
\begin{align}
\xi_{n-1}=u\left(\lambda+\sum_{i=0}^{n}S_{i}(n)(C_1,\cdots, C_{n-1})v^{i}\right).
\end{align}
For convenience, from now to the end of proof, we will omit the variables $C_1, \cdots, C_{n-1}$ and simply write the polynomial $S_{i}(j)(C_1,\cdots, C_{j-1})$ as $S_{i}(j)$.

By Lemma \ref{lemmavj}, there exists a positive integer $N_1=N_1(q)$ such that
\begin{align*}
v_{k-N_1}>v_{k-N_1-1}>\cdots >v_{0}.
\end{align*}
Note that \eqref{Jacobi} implies $P_0>0$. If $\lambda-C_{n}<0$, by \eqref{lambda-eq} we know there exists a positive integer $N_2=N_2(q)$ such that for any integers $m\ge N_2$,
\begin{align}\label{less-than-0}
\lambda \overline{P}_{0,2m-1}+\sum_{i=0}^{n}S_{i}(n)\overline{P}_{i,2m-1}<0.
\end{align}
Let $N>\max\{N_1,N_2\}$. Using \eqref{v-new} and by direct calculations, we find that
\begin{align}
\sum_{j=1}^{2N-1}(-1)^{j-1}v_{k-j}&=(k+\Lambda_{n-1}(k^{-1})k^{-1})^{k-2}\frac{1}{k!}q^{-k(k-1)/2} \nonumber \\
&\Bigg(\sum_{\ell=1}^{n-1}\Big(C_{\ell}\overline{P}_{0,2N-1}+\sum_{i=0}^{\ell}S_{i}(\ell)\overline{P}_{i,2N-1}\Big)k^{-(\ell-1)}  \nonumber\\
&+\Big(\lambda\overline{P}_{0,2N-1}+\sum_{i=0}^{n}S_{i}(n)\overline{P}_{i,2N-1} \Big)k^{-(n-1)}+O(k^{-n}) \Bigg). \label{start}
\end{align}
Now for each $1\le \ell \le n-1$, from \eqref{lambda-eq} we deduce that
\begin{align}
C_{\ell}\overline{P}_{0}+\sum_{i=0}^{\ell}S_{i}(\ell)\overline{P}_{\ell}=0.
\end{align}
Hence
\begin{align}
&\left|\left(C_{\ell}\overline{P}_{0,2N-1}+\sum_{i=0}^{\ell}S_{i}(\ell)\overline{P}_{i,2N-1}\right)k^{-(\ell-1)} \right|\nonumber \\
=&\left|\sum_{j\ge 2N}\Big(C_{\ell}+\sum_{i=0}^{\ell}S_{i}(\ell)v^{i}\Big)(-1)^{j-1}(2j-1)q^{j(j-1)/2}\right| \nonumber \\
\le& |C_{\ell}|\left|\sum_{j\ge 2N}(-1)^{j-1}(2j-1)q^{j(j-1)/2} \right|\nonumber \\
&+\sum_{i=0}^{\ell}\left|S_{i}(\ell)\right|\left|\sum_{j\ge 2N}(-1)^{j-1}(2j-1)(j(j-1))^{i}q^{j(j-1)/2} \right| \nonumber \\
\le &|C_{\ell}|(4N-1)q^{N(2N-1)}+\sum_{i=0}^{\ell}|S_{i}(\ell)|(4N-1)(2N(2N-1))^{i}q^{N(2N-1)} \nonumber \\
=&O(q^{N^2}). \nonumber
\end{align}
Here for the last inequality, we have used the fact that when $N$ is sufficiently large, the sequence $(2j-1)(j(j-1))^iq^{j(j-1)/2}$ will be decreasing when $j\ge 2N$.

Note that when $k\le q^{-N^2/n}$, we have $q^{N^2}\le k^{-n}$.
Hence for each $k$ satisfying $q^{-3N}\le k\le q^{-N^2/n}$, \eqref{start} implies
\begin{align}
\sum_{j=1}^{2N-1}(-1)^{j-1}v_{k-j}=&(k+\Lambda_{n-1}(k^{-1})k^{-1})^{k-2}\frac{1}{k!}q^{-k(k-1)/2}\Big(\lambda\overline{P}_{0,2N-1}\nonumber\\
&+\sum_{i=0}^{n}S_{i}(n)\overline{P}_{i,2N-1} \Big)k^{-(n-1)}+O(k^{-n}) \Big).
\end{align}
From \eqref{less-than-0}, when $N$ is large enough, and for $k$ satisfying $q^{-3N}\le k \le q^{-N^2/n}$, we can guarantee that
\begin{align}\label{mid}
\Big(\lambda\overline{P}_{0,2N-1}+\sum_{i=0}^{n}S_{i}(n)\overline{P}_{i,2N-1} \Big)k^{-(n-1)}+O(k^{-n}) <0.
\end{align}
Therefore, for such $k$ and $N$, we have
\begin{align}
\sum_{j=1}^{2N-1}(-1)^{j-1}v_{k-j}<0
\end{align}
and by Lemma \ref{lemmavj}
\begin{align}
\sum_{j=2N}^{k}(-1)^{j-1}v_{k-j}<v_{k-2N-1}-v_{k-2N}<0.
\end{align}
So we have
\begin{align}\label{end-0}
(-1)^{k}\sum_{n=0}^{2k-1}(-1)^nu_{n}=\sum_{j=1}^{k}(-1)^{j-1}v_{k-j}<0.
\end{align}
By \eqref{mid} we know that there exists a constant $c>0$ such that for $k$ large enough,
\begin{align}\label{end-1}
\left|\sum_{n=0}^{2k-1}(-1)^nu_{n} \right|&>c\frac{(k+\Lambda_{n-1}(k^{-1})k^{-1})^{k-2}}{k!}q^{-k(k-1)/2}k^{-(n-1)}\nonumber \\
&>\frac{(k+\Lambda_{n-1}(k^{-1})k^{-1})^{2k}}{(2k)!}q^{k}=u_{2k}.
\end{align}
Moreover, since $u_n$ is decreasing when $n>k$,
\begin{align}\label{end-2}
u_{2k}>\left|\sum_{n=2k}^{\infty}u_{n}(-1)^n \right|.
\end{align}
From \eqref{end-0}, \eqref{end-1} and \eqref{end-2}, we deduce that
\begin{align*}
(-1)^kf(-(k+\Lambda_{n-1}(k^{-1})k^{-1})q^{1-k})=\sum_{n=0}^{\infty}(-1)^nu_n<0
\end{align*}
for large $k$.

Similarly, if $\lambda-C_{n}>0$, we  have for large $k$,
\begin{align*}
(-1)^kf(-(k+\Lambda_{n-1}(k^{-1})k^{-1})q^{1-k})>0.
\end{align*}
\end{proof}

\begin{proof}[Proof of Theorem \ref{thm1}]
For any $n\ge 1$, we choose $\lambda'<C_{n}<\lambda''$. Let
\begin{align*}
\xi_{k}'=-kq^{1-k}\left(1+\sum_{i=1}^{n-1}C_{i}k^{-1-i}+\lambda'k^{-1-n}\right), \\
\xi_{k}''=-kq^{1-k}\left(1+\sum_{i=1}^{n-1}C_{i}k^{-1-i}+\lambda''k^{-1-n}\right).
\end{align*}
By Lemma \ref{sign} we have $f(\xi_{k}')f(\xi_{k}'')<0$. Therefore, by the Intermediate Value Theorem, there exists a root in the interval $(\xi_{k}',\ \xi_{k}'')$. Thanks to \eqref{ratio}, when $k$ is large enough, we know this interval contains only one root and this root must be $x_{k}$ (see also \cite[Proof of Theorem 1]{Zhang}).
Thus we can write the root as
\begin{align}
x_{k}=-kq^{1-k}\left(1+\sum_{i=1}^{n-1}C_{i}k^{-1-i}+\theta_{n}(k)k^{-1-n}\right).
\end{align}
By letting $\lambda'$ and $\lambda''$ tend to $C_{n}$ from the left side and right side, respectively, we see that we must have
\begin{align}
\lim\limits_{k\rightarrow \infty}\theta_{n}(k)=C_{n}.
\end{align}
Thus $\theta_{n}(k)=C_{n}+o(1)$ as $k$ tends to infinity. This means
\begin{align}
x_{k}=-kq^{1-k}\left(1+\sum_{i=1}^{n-1}C_{i}k^{-1-i}+C_{n}k^{-1-n}+o(k^{-1-n})\right).
\end{align}
This proves \eqref{xk-thm} for any $n\ge 1$.

From the definition \eqref{rec} and Lemma \ref{P-lemma}, it is clear that $C_{i}$ is a multivariate polynomial of $A_0$, $A_1$, $\dots$, $A_{i-1}$ with rational coefficients.
\end{proof}

\section{Representations of $C_n$}\label{representation}
\subsection{Representation of $C_{n}$ using $A_0,A_1,\cdots, A_{n-1}$}
We have seen in \eqref{rec} that
\begin{equation}\label{cn}
C_n=-S_0(n)+6S_1(n)P_1+12S_2(n)P_2+...+3\cdot 2^nS_n(n)P_n,
\end{equation}
where each $P_i$ is given by a recursive formula in Lemma \ref{P-lemma}, and each $S_i(n)$ can be determined from  \eqref{sirepa} by setting $a_i=C_{i+1}$. Indeed, for $i=0,1,...,n$,
\begin{equation}\label{sirep}
\begin{aligned}
&S_i(n)=[uv^i]\Delta(n+1,0)\\
&+\sum_{N=3}^{n+1}\sum_{m=1}^{[N/2]}\sum_{i_1+...+i_m=n-N+m+1}([uv^i]\Delta(N,m)+[uv^i]\Delta(N-2,m-1))C_{i_1}\cdot\cdot\cdot C_{i_m},
\end{aligned}
\end{equation}
where $\Delta(N,m)$ was given in \eqref{Delta-defn}, and $[uv^i]\Delta(N,m)$ means the coefficient of the term $uv^i$ in the representation of $\Delta(N,m)$ in Proposition \ref{Delta-uv-exp}.
\begin{proposition}\label{sn}
For $n\ge 1$,
\[S_n(n)=\frac{(-1)^{n-1}}{3\cdot 2^n(n-1)!}.\]
\end{proposition}
\begin{proof}
By Lemma \ref{sigmaQ},
\[G(N,0)=Q_N=\frac{(-1)^N}{2^NN!}j^{2N}+\frac{(-1)^N(2N-5)}{3\cdot 2^N(N-1)!}j^{2N-1}+O(j^{2N-2}),\]
\[H(N,0)=(-1)^N\sigma_N=\frac{(-1)^N}{2^NN!}j^{2N}-\frac{(-1)^N(2N+1)}{3\cdot 2^N(N-1)!}j^{2N-1}+O(j^{2N-2}).\]
Then
\begin{align}\label{deltaN0}\Delta(N,0)=\frac{(-1)^N}{3\cdot 2^{N-2}(N-2)!}j^{2N-1}+O(j^{2N-2}).\end{align}
So
\begin{equation}\label{delta}
\Delta(n+1,0)=\frac{(-1)^{n+1}}{3\cdot 2^n(n-1)!}uv^n+\ {\rm lower\ degree\ terms}.
\end{equation}
Thanks to \eqref{deltaNm}, we know that in \eqref{sirep}, $\Delta(n+1,0)$ is the only term containing $uv^n$, so $S_n(n)$ is exactly the coefficient of $uv^n$ in \eqref{delta}.
\end{proof}
\begin{proposition}\label{s0}
For $n\ge 3$,
\[S_0(n)=\sum_{i_1+i_2=n-1}C_{i_1}C_{i_2}. \]
\end{proposition}
\begin{proof}
Recall that $S_0(n)=0$ if $n=1,\ 2$. From Proposition \ref{Delta-uv-exp} we know that when $n\ge3$, any other $\Delta(N,m)$ with $(N,m)\ne(2,1)$ in \eqref{sirep} must have the factor $uv$. Recall that $\Delta(2,1)=u$. So \[S_0(n)=[u]\Delta(2,1)\sum_{i_1+i_2=n-1}C_{i_1}C_{i_2}=\sum_{i_1+i_2=n-1}C_{i_1}C_{i_2}.\]
\end{proof}

\begin{proposition}\label{si}
For $1\le i\le n-1$, $S_{i}(n)$ is a polynomial of $C_1,...,C_{n-i}$ with degree  $\le\min\{[\frac{2(n-i)+1}3],[\frac{n+1}2]\}$. Moreover, this polynomial has the form
\begin{align}\label{cnrep}\tilde F(C_1,...,C_{n-i-1})+\frac{(-1)^{i}}{2^{i}i!}C_{n-i},\end{align}
when $\tilde F$ is a polynomial depending on $n$ and $i$.
\end{proposition}
\begin{proof}
To show that the subscripts of $C_k$'s in $S_i$ are at most $n-i$, we need to analyze the terms associated with each $[uv^i]\Delta(N,m)$ in \eqref{sirep}.  Recall that \eqref{deltaNm} gives the restriction $[\frac{2N-3m-1}2]\ge i$ on these $(N,m)$.  We consider two different cases. When $m\ge1$, we have $N-2\ge[\frac{2N-3m-1}2]\ge i$, so $N\ge i+2$. The term associated with $[uv^i]\Delta(N,m)$ is
\begin{equation}\label{degree}
\sum_{i_1+...+i_m=n-N+m+1}C_{i_1}...C_{i_m}+\sum_{i_1+...+i_{m+1}=n-N+m}C_{i_1}...C_{i_{m+1}}.
\end{equation}
(Note that when $N=n,\ n+1$, the second sum vanishes.) So the maximal subscript \[\max_li_l\le n-N_{\min}+2=n-i.\] When $m=0$, the term associated with $\Delta(N,0)$ is $C_{n-N}$ if $N\le n-1$, and is a constant if $N=n,\ n+1$. Since $N-1=[\frac{2N-3m-1}2]\ge i$, we have $N\ge i+1$. So the subscript \[n-N\le n-N_{\min}=n-i-1<n-i.\] Consequently,  $S_{i}(n)$ contains only $C_1,...,C_{n-i}$.

Since $[\frac{2N-3m-1}2]\ge i$, we have $m\le [\frac{2N-2i-1}3]$. The degree of \eqref{degree}, as a polynomial of $C_1,...,C_{n-i}$,  is $m+1$ when $N\le n-1$, and is $m$ when $N=n,\ n+1$. When $N\le n-1$,
\[m+1\le[\tfrac{2(n-1)-2i-1}3]+1=[\tfrac{2(n-i)}3].\]
When $N=n,\ n+1$,
 \[m\le[\tfrac{2(n+1)-2i-1}3]=[\tfrac{2(n-i)+1}3].\]
Recall that $m\le [N/2]$. So the degree of $S_{i}(n)$, as a polynomial of $C_1,...,C_{n-i}$, is  at most $\min\{[\frac{2(n-i)+1}3],[\frac{n+1}2]\}$.

The coefficient of $C_{n-i}$ in \eqref{sirep} is $[uv^i]\Delta(i+2,1)+[uv^i]\Delta(i,0)$. By \eqref{GNm}, \eqref{HNm}, and \eqref{delta}, we have
\[\Delta(i+2,1)=\frac{(-1)^{i}}{2^{i}i!}uv^{i}+\ {\rm lower\ degree\ terms},\]
\[\Delta(i,0)=\frac{(-1)^i}{3\cdot 2^{i-1}(i-2)!}uv^{i-1}+\ {\rm lower\ degree\ terms},\]
which implies $[uv^i]\Delta(i+2,1)+[uv^i]\Delta(i,0)=\frac{(-1)^{i}}{2^{i}i!}$.
\end{proof}
\begin{proposition}\label{Cn-poly-prop}
For any $n\ge 1$,
\begin{align}\label{add-Cn}
C_n=F(A_0,...,A_{n-2})+\frac{(-1)^{n-1}}{(n-1)!}A_{n-1},
\end{align}
where $F$ is a polynomial depending on $n$ and has degree at most $n$.
\end{proposition}
\begin{proof}
From the definition of $P_n$ in Lemma \ref{P-lemma}, one can prove by induction that for any $n\ge 1$, the degree of the multivariate polynomial $P_n$ is $n$, and
\begin{align}\label{Pnrep}
P_n=A_{n-1}+ {\rm higher\ degree\ terms} \ ({\rm  without}\ A_{n-1}).
\end{align}
So the coefficient of $A_{n-1}$ in $C_n$ is exactly $\frac{(-1)^{n-1}}{(n-1)!}$ by \eqref{cn} and Proposition \ref{sn}. To show the degree of $C_n$ is at most $n$, we use induction. Suppose deg$(C_l)\le l$, $1\le l<n$. Then one can estimate the degree of $C_n$ directly by \eqref{cn} and \eqref{sirep}. For example, the term \[[uv^i]\Delta(N,m)P_i\sum_{i_1+...+i_m=n-N+m+1}C_{i_1}\cdot\cdot\cdot C_{i_m}\]
has degree at most
\[\Big[\frac{2N-3m-1}2\Big]+(n-N+m+1)=\Big[n-\frac12m+\frac12\Big]\le n.\]
Other terms can be estimated similarly. So  deg$(C_n)\le n$.
\end{proof}

\subsection{The linear terms in the representation of $C_n$}
We have known that $C_n$ is a multivariate polynomial of $A_0,...,A_{n-1}$, which has the form \eqref{add-Cn} and can be determined recursively by \eqref{cn}. However, from the examples presented in Section 1, we see that this multivariate polynomial may have a very complicated structure, since it contains nonlinear terms as well as linear terms. In this section, we will see that at least the linear terms can be understood.

For $i=1,...,n$, let $S_{i,0}(n)$ be the constant term in $S_i(n)$. By \eqref{sirep}, we have
\[\Delta(n+1,0)=Q_{n+1}-(-1)^{n+1}\sigma_{n+1}=S_{n,0}(n)uv^n+...+S_{2,0}(n)uv^2+S_{1,0}(n)uv.\]
Then by \eqref{Pnrep} and \eqref{cn}, the coefficient of the linear term $A_{i-1}$ in $C_n$ equals to
\[3\cdot2^iS_{i,0}(n).\]
We have obtained the coefficient of the linear term $A_{n-1}$ in Proposition \ref{Cn-poly-prop}. However, to determine the explicit formulas of the coefficients for other linear terms, we have to obtain the explicit expansion of $\sigma_{n+1}(j)$ (i.e., the unsigned Stirling number of the first kind $c(j,j-n-1)$). Although it is possible to determine the first several terms in the expansions (see Lemma \ref{sigmaQ} for the first two terms), complete expansions are difficult and unknown.  Fortunately, the sum of these coefficients has simple closed form, which gives the coefficient of $q$ in the expansion of $C_n(q)$ in $q$ (see Remark \ref{remark1}).

\begin{proposition}\label{sumcoeff}
 The sum of the coefficients of the linear terms $A_0,A_1,...,A_{n-1}$ in $C_n$ equals to $(-1)^{n-1}$.
\end{proposition}
\begin{proof}
The sum of the coefficients of the linear terms $A_0,A_1,...,A_{n-1}$ in $C_n$ equals to
\begin{equation}\label{sum}
6S_{1,0}(n)+12S_{2,0}(n)+...+3\cdot 2^{n}S_{n,0}(n).
\end{equation}
Note that
\begin{equation}\label{deltauv}
\Delta(n+1,0)=S_{n,0}(n)uv^n+...+S_{2,0}(n)uv^2+S_{1,0}(n)uv.
\end{equation}
To compute the value of \eqref{sum}, we only need to set $u=3$ and $v=2$, namely $j=2$ in \eqref{deltauv}. Note that when $j=2$
\[G(N,0)=(-1)^N,\ \  H(N,0)=0, \ \ N=2,3,...,\]
So $\Delta(n+1,0)=G(n+1,0)-H(n+1,0)=(-1)^{n+1}$.\end{proof}

Although it is difficult to find $S_{i,0}(n)$ for all $0\le i \le n$, we are able to find explicit formulas for $S_{1,0}(n)$ and $S_{2,0}(n)$, which give us explicit formulas for the coefficients of the linear terms $A_0$ and $A_1$. These results are useful in proving Theorem \ref{thm2}.
\begin{proposition}\label{S10-coeff}
For any $n\ge 1$,
\begin{align}
S_{1,0}(2n)=0, \quad S_{1,0}(2n-1)=\frac{B_{2n}}{n}, \quad S_{2,0}(2n)=-\frac{B_{2n}}{2n}.
\end{align}
\end{proposition}
\begin{proof}
Since  $\Delta(n,0)=Q_{n}-(-1)^{n}\sigma_{n}$, we have $Q_{n}=\Delta(n,0)+(-1)^{n}\sigma_{n}$. By \eqref{Q-rec} we have
\begin{align*}
(-1)^n\sigma_{n}+\Delta(n,0)=-\sum_{m=0}^{n-1}\sigma_{n-m}\left((-1)^m\sigma_{m}+\Delta(m,0) \right).
\end{align*}
This implies
\begin{align}\label{Delta}
\Delta(n,0)=-\sum_{m=0}^{n-1}\Delta(m,0)\sigma_{n-m}-\sum_{m=0}^{n}(-1)^{m}\sigma_{m}\sigma_{n-m}.
\end{align}
Note that for any $n\ge 1$, $\sigma_{n}$ is divisible by $v=j(j-1)$. Since $\Delta(0,0)=\Delta(1,0)=0$,  from \eqref{Delta}, we know $\Delta(n,0)$ is divisible by $v$.
Replacing $n$ by $2n+1$ in \eqref{Delta} and observing that
\begin{align*}
\sum_{m=0}^{2n+1}(-1)^{m}\sigma_{m}\sigma_{2n+1-m}=\sum_{m=0}^{n}\left((-1)^{m}\sigma_{m}\sigma_{2n+1-m}+(-1)^{2n+1-m}\sigma_{2n+1-m}\sigma_{m} \right)=0,
\end{align*}
we obtain
\begin{align}\label{Delta-odd}
\Delta(2n+1,0)=-\sum_{m=0}^{2n}\Delta(m,0)\sigma_{2n+1-m}.
\end{align}
Since both $\Delta(m,0)$ and $\sigma_{2n+1-m}$ are divisible by $v$, we know that $\Delta(2n+1,0)$ is divisible by $v^2$. Therefore, from \eqref{deltauv} we know $S_{1,0}(2n)=0$.

Furthermore, from Newton's identities, we have
\begin{align}\label{sigma-p}
m\sigma_{m}=\sum_{k=1}^{m}(-1)^{k-1}p_{k}\sigma_{m-k}.
\end{align}
For each $1\le k \le m-1$, $p_{k}\sigma_{m-k}$ is divisible by $v^2$. It is well known that if $m\ge 3$ is odd, then $p_{m}$ is divisible by $v^2$. Hence \eqref{sigma-p} implies that $\sigma_{2k+1}$ is divisible by $v^2$ for any $k\ge 1$. Now we compare the coefficients of $uv^2$ in both sides of \eqref{Delta-odd}. For the left hand side, it is clearly equal to $S_{2,0}(2n)$. For the right hand side, if $0\le m \le 2n-1$ is odd, then $\Delta(m,0)$ is divisible by $v^2$ and $\sigma_{2n+1-m}$ is divisible by $v$.  If $0\le m \le 2n-1$ is even, then $\Delta(m,0)$ is divisible by $v$ and $\sigma_{2n+1-m}$ is divisible by $v^2$. Hence for any $0\le m \le 2n-1$, $\Delta(m,0)\sigma_{2n+1-m}$ is always divisible by $v^3$. Thus the term $uv^2$ only appears in $\Delta(2n,0)\sigma_1$, and hence equals to $-\frac{1}{2}S_{1,0}(2n-1)uv^2$. Thus we obtain
\begin{align}\label{S20-S10-odd-relation}
S_{2,0}(2n)=-\frac{1}{2}S_{1,0}(2n-1).
\end{align}

Now we determine $S_{1,0}(2n-1)$.
Replacing $n$ by $2n$ in \eqref{Delta}, we obtain
\begin{align}\label{Delta-even}
\Delta(2n,0)=-2\sigma_{2n}-\sum_{m=1}^{2n-1}(-1)^{m}\sigma_{m}\sigma_{2n-m}-\sum_{m=0}^{2n-1}\Delta(m,0)\sigma_{2n-m}.
\end{align}
Comparing the coefficient of $j$ on both sides, we obtain
\begin{align}\label{S10-sigma}
S_{1,0}(2n-1)=-2[j]\sigma_{2n}.
\end{align}

From \eqref{p-m} we know that
\begin{align}\label{j-pm}
[j]p_{m}=(-1)^mB_{m}, \quad m\ge 2.
\end{align}
From \eqref{sigma-p} we get
\begin{align}
m[j]\sigma_{m}=-B_{m}, \quad m \ge 2.
\end{align}
Replacing $m$ by $2m$, we obtain
\begin{align}\label{j-sigma}
[j]\sigma_{2m}=-\frac{B_{2m}}{2m}, \quad m\ge 1.
\end{align}
Substituting \eqref{j-sigma} into \eqref{S10-sigma}, we complete the proof.
\end{proof}

\begin{proposition}\label{S20-coeff}
For any $n\ge 2$,
\begin{align}
S_{2,0}(2n-1)=-\frac{3B_{2n}}{n}.
\end{align}
\end{proposition}
\begin{proof}
From Proposition \ref{S10-coeff} we know it suffices to show that
\begin{align}\label{S20-S10-relation}
S_{2,0}(2n-1)+3S_{1,0}(2n-1)=0.
\end{align}
We observe that
\begin{align*}
[j^2]\left(a_{2}uv^2+a_{1}uv\right)=[j^2]\left(a_{2}j^2(j-1)^2(2j-1)+a_{1}j(j-1)(2j-1)\right)=-a_2-3a_1.
\end{align*}
Comparing the coefficients of $j^2$ on both sides of \eqref{Delta-even}, we deduce that
\begin{align}
&-S_{2,0}(2n-1)-3S_{1,0}(2n-1)\nonumber \\
=&-2[j^2]\sigma_{2n}-\sum_{m=1}^{2n-1}(-1)^m[j]\sigma_{m}\cdot [j]\sigma_{2n-m}-\sum_{m=0}^{2n-1}S_{1,0}(m-1)[j]\sigma_{2n-m} \nonumber \\
=&-2[j^2]\sigma_{2n}-\sum_{m=1}^{n-1}[j]\sigma_{2m}[j]\sigma_{2n-2m}-\sum_{m=1}^{n-1}S_{1,0}(2m-1)[j]\sigma_{2n-2m} \nonumber\\
=&-2[j^2]\sigma_{2n}+\sum_{m=1}^{n-1}[j]\sigma_{2m}\cdot [j]\sigma_{2n-2m},
\end{align}
where in the last equality we used \eqref{S10-sigma}. Hence the proposition is equivalent to the assertion that for any $n\ge 2$,
\begin{align}\label{j2-convo}
2[j^2]\sigma_{2n}=\sum_{m=1}^{n-1}[j]\sigma_{2m}\cdot [j]\sigma_{2n-2m}.
\end{align}
From \eqref{p-m} we deduce that
\begin{align*}
[j^2]p_{m}=\frac{(-1)^{m-1}}{2}mB_{m-1}.
\end{align*}
Hence for $m\ge 2$,
\begin{align}\label{j2-pm}
[j^2]p_{2m}=0.
\end{align}
From \eqref{sigma-p} we deduce that for $n\ge 2$,
\begin{align*}
2n[j^2]\sigma_{2n}&=-[j^2]p_{2n}+\sum_{k=2}^{2n-1}(-1)^{k-1}[j]p_{k}[j]\sigma_{2n-k}.
\end{align*}
Using \eqref{j-pm}, \eqref{j-sigma} and \eqref{j2-pm},  we obtain
\begin{align}\label{j2-last}
[j^2]\sigma_{2n}=\frac{1}{2n}\sum_{k=1}^{n-1}\frac{B_{2k}B_{2n-2k}}{2n-2k}.
\end{align}
From \eqref{j-sigma} we have
\begin{align}\label{j-convo}
&\sum_{m=1}^{n-1}[j]\sigma_{2m}\cdot [j]\sigma_{2n-2m}\nonumber \\
=&\sum_{m=1}^{n-1}\frac{B_{2m}}{2m}\cdot \frac{B_{2n-2m}}{2n-2m} \nonumber\\
=&\sum_{m=1}^{n-1}B_{2m}B_{2n-2m}\frac{1}{2n}\left(\frac{1}{2m}+\frac{1}{2n-2m} \right) \nonumber \\
=&\frac{1}{n}\sum_{m=1}^{n-1}\frac{B_{2m}B_{2n-2m}}{2n-2m}.
\end{align}
Comparing \eqref{j2-last} and \eqref{j-convo}, we complete the proof of \eqref{j2-convo} and the proposition.
\end{proof}

\subsection{Alternative representations of $C_{n}$}
The representations of $C_{n}$ are not unique. Indeed, it is possible to represent $C_{n}$ using only $A_{0}$, $A_1$ and $A_2$. For this we need to know the relation between $A_3$ and $A_0,A_1,A_2$.
Recall that $E_2, E_4$ and $E_6$ denotes three Eisenstein series as given in \eqref{P-defn}--\eqref{R-defn}.
The following identities of Ramanujan are well-known (see \cite[Theorem 4.2.3]{Berndt}, for example.):
\begin{align}
\Theta(E_2)&=\frac{E_{2}^2-E_{4}}{12}, \label{Rama-1}\\
\Theta(E_4)&=\frac{E_2E_4-E_6}{3}, \label{Rama-2}\\
\Theta(E_6)&=\frac{E_2E_6-E_{4}^2}{2}. \label{Rama-3}
\end{align}
We first express $E_2$, $E_4$ and $E_6$ in terms of $A_0$, $A_1$ and $A_2$.
\begin{proposition}\label{PQR-A}
We have
\begin{align}
E_2&=1-24A_0, \label{E2-exp}\\
E_4&=1-48A_0+576A_0^2+288A_1, \label{E4-exp} \\
E_6&=1-72A_0+1728A_0^2-13824A_0^3+432A_1-10368A_0A_1-864A_2. \label{E6-exp}
\end{align}
\end{proposition}
\begin{proof}
The relation \eqref{E2-exp} follows from definition.

Since $\Theta(A_0)=A_1$, \eqref{Rama-1} implies
\begin{align*}
E_4=E_{2}^2-12\Theta(E_2)=E_{2}^2+288A_1=1-48A_0+576A_0^2+288A_1.
\end{align*}
This proves \eqref{E4-exp}.

Next, we have
\begin{align}\label{P2}
\Theta(E_{2}^2)=\Theta(1-48A_0+576A_0^2)=-48A_1+1152A_0A_1.
\end{align}
Applying $\Theta$ to both sides of \eqref{Rama-1}, by \eqref{Rama-2} and \eqref{P2} we obtain
\begin{align*}
\Theta^2(E_2)&=\frac{1}{12}\Theta(E_{2}^2-E_4)\\
&=\frac{1}{12}\Theta(E_{2}^2)-\frac{1}{36}(E_2E_4-E_6) \\
&=96A_0A_1-4A_1-\frac{1}{36}(E_2E_4-E_6).
\end{align*}
On the other hand, we have
\begin{align}
\Theta^2(E_2)=-24\Theta^2(A_0)=-24A_2.
\end{align}
So we deduce that
\begin{align}\label{A2-exp}
A_2=-4A_0A_1+\frac{1}{6}A_1+\frac{1}{864}(E_2E_4-E_6).
\end{align}
This implies
\begin{align}\label{E6-exp-1}
E_6=E_2E_4-3456A_0A_1+144A_1-864A_2.
\end{align}
Substituting \eqref{E2-exp} and \eqref{E4-exp} into \eqref{E6-exp-1} and simplifying, we obtain
\eqref{E6-exp}.
\end{proof}
From Proposition \ref{PQR-A}, it is easy to express $A_0,A_1$ and $A_2$ as polynomials in $E_2,E_4$ and $E_6$.
\begin{corollary}\label{E-A-repn}
We have
\begin{align*}
A_0&=\frac{1}{24}(1-E_2), \\
A_1&=\frac{1}{288}(E_4-E_{2}^2), \\
A_2&=-\frac{1}{1728}(E_{2}^3-3E_2E_4+2E_6).
\end{align*}
\end{corollary}

\begin{lemma}\label{An-poly-lem}
For any $n\ge 1$, $A_n$ can be written as a multivariate polynomial in $A_0$, $A_1$ and $A_2$ with integer coefficients and degree at most $n-1$. In particular, we have
\begin{align}\label{A3-poly}
A_3=A_2+36A_1^2-24A_0A_2.
\end{align}
\end{lemma}
\begin{proof}
Applying the $\Theta$ operator to both sides of \eqref{A2-exp}, upon using \eqref{Rama-1}--\eqref{Rama-3} and simplifying, we obtain
\begin{align*}
A_3&=-4\Theta(A_0A_1)+\frac{1}{6}\Theta(A_1)+\frac{1}{864}\Theta(E_2E_4-E_6) \\
&=-4A_0A_2-4A_1^2+\frac{1}{6}A_2+\frac{1}{864}\left(\Theta(E_2)E_4+E_2\Theta(E_4)-\Theta(E_6) \right) \\
&=-4A_0A_2-4A_1^2+\frac{1}{6}A_2+\frac{5}{10368}\left(E_{2}^2E_4+E_{4}^2-2E_{2}E_{6} \right).
\end{align*}
Now substituting \eqref{E2-exp}--\eqref{E6-exp} into the above identity and simplifying, we obtain \eqref{A3-poly}.

Since $A_{m+1}=\Theta(A_m)$, the first assertion follows by using \eqref{A3-poly} and induction on $m$.
\end{proof}
Finally, we present a proof of Theorem \ref{thm2}.
\begin{proof}[Proof of Theorem \ref{thm2}]
From Lemma \ref{An-poly-lem} and Proposition \ref{Cn-poly-prop}, we know that $C_{n}$ can be represented as a polynomial of $A_0, A_1$ and $A_2$. For the uniqueness,  it is known that $E_2$, $E_4$ and $E_6$ are algebraically independent over $\mathbb{C}$ (see \cite[Lemma 117]{Martin}, for example). Therefore, Corollary \ref{E-A-repn} implies that $A_0$, $A_1$ and $A_2$ are also algebraically independent over $\mathbb{C}$. Hence the expression of $C_{n}$ as a polynomial in $A_0$, $A_1$ and $A_2$ is unique. From Theorem \ref{thm1} and Lemma \ref{An-poly-lem}, it is easy to see that all the coefficients are rational numbers.

From Lemma \ref{An-poly-lem}, we see that $A_n=A_2+{\rm higher\ degree\ terms}$ for all $n\ge 3$. So Proposition \ref{sumcoeff} still holds for this representation of $C_n$ in $A_0, A_1$ and $A_2$, and the coefficients of $A_0$ and $A_1$ do not change. Therefore, the coefficients of the linear terms $A_0$, $A_1$, $A_2$ in this representation of $C_n$ are $6S_{1,0}(n)$, $12S_{2,0}(n)$, $(-1)^{n-1}-6S_{1,0}(n)-12S_{2,0}(n)$, respectively. Since $S_{1,0}(n)$ and $S_{2,0}(n)$ are given explicitly in Propositions \ref{S10-coeff} and \ref{S20-coeff}, we complete our proof of Theorem \ref{thm2}.
\end{proof}

\bibliographystyle{plain}

\begin{thebibliography}{0}

\bibitem{Berndt}
B.C. Berndt, Number Theory in the Spirit of Ramanujan, Amer. Math.
Soc., Providence, 2006.


\bibitem{DGT} G. Derfel, P. J. Grabner and R. P. Tichy, On the asymptotic behaviour of the zeros of solutions of the functional-differential equation with rescaling, arXiv:1612.06226.

\bibitem{morris} A. Feldstein, G. Morris and E. Bowen, The phragm\'en-lindel\''of principle and a class of
functional differential equations, in:  Ordinary Differential Equations, Academic Press, 1972.


\bibitem{Grabner} P. J. Grabner and B. Steinsky, Asymptotic behaviour of the poles of a special generating function for
acyclic digraphs, Aequationes Math. 70 (2005), 268--278.

\bibitem{iserles}  A. Iserles, On the generalized pantograph functional-differential equation, Eur. J.  Appl. Math. 4(01) (1993), 1--38.

\bibitem{langley} J.K. Langley, A certain functional--differential equation, J. Math. Anal. Appl.  244(2) (2000), 564--567.


\bibitem{liu} Y. Liu, On some conjectures by Morris et al. about zeros of an entire function, J. Math.
Anal. Appl., 226(1) (1998), 1--5.


\bibitem{Martin} F. Martin and E. Royer, Formes modulaires et p\'eriodes,  Vol. 12. S¨¦min. Congr., Soc.
Math. France, Paris, 2005.


\bibitem{Robinson} R. W. Robinson, Counting labeled acyclic digraphs, In: New Directions in the Theory of Graphs
(ed., F. Harari), 239--279, Academic Press, New York, 1973.


\bibitem{scott2005} A. Scott and A. Sokal, The repulsive lattice gas, the independent-set polynomial, and the
lov¨¢sz local lemma, J.  Stat. Phys. 118(5-6) (2005), 1151--1261.


\bibitem{sokal2009} A. Sokal, Some wonderful conjectures (but almost no theorems) at the boundary between
analysis, combinatorics and probability.\\
\url{http://ipht.cea.fr/statcomb2009/misc/Sokal_20091109.pdf}

\bibitem{sokal2012} A. Sokal, The leading root of the partial theta function, Adv. Math. 229 (5) (2012), 2603--2621.



\bibitem{Zhang} C. Zhang, An asymptotic formula for the zeros of the deformed exponential function, J. Math. Anal. Appl. 441 (2016), 565--573.
\end{thebibliography}

\end{document}